\documentclass[12pt]{article}
\usepackage[utf8]{inputenc}

\usepackage{color}

\usepackage{epsfig}
\usepackage{cite}
\usepackage{latexsym,amsmath}
\usepackage{amsfonts}
\usepackage{float}
\usepackage{pict2e}
\usepackage{tikz}
\usepackage{caption}
\usepackage{amsfonts}
\usepackage{amssymb}
\usepackage{mathrsfs}
\usepackage{amsmath}
\usepackage{enumerate}
\usepackage{graphicx}
\usepackage{subfig}
\usepackage{MnSymbol}
\usepackage{mathtools}
\usepackage{amsmath}

\usepackage{amssymb}
\begin{document}
	
	\bibliographystyle{plain}
	
	\pagestyle{myheadings}
	\thispagestyle{empty}
	\newtheorem{theorem}{Theorem}
	\newtheorem{corollary}[theorem]{Corollary}
	\newtheorem{definition}{Definition}
	\newtheorem{guess}{Conjecture}
	\newtheorem{claim}{Claim}
	\newtheorem{problem}{Problem}
	\newtheorem{question}{Question}
	\newtheorem{lemma}[theorem]{Lemma}
	\newtheorem{proposition}[theorem]{Proposition}
	\newtheorem{observation}[theorem]{Observation}
	\newenvironment{proof}{\noindent {\bf
			Proof.}}{\hfill\rule{3mm}{3mm}\par\medskip}
	\newcommand{\remark}{\medskip\par\noindent {\bf Remark.~~}}
	\newcommand{\pp}{{\it p.}}
	\newcommand{\de}{\em}
	\newtheorem{example}{Example}

	\title{\bf The fractional chromatic number of generalized cones over graphs}
	
	\author{Jialu Zhu\thanks{Department of Mathematics, Zhejiang Normal University, Email: 709747529@qq.com, }         \and
		Xuding Zhu\thanks{Department of Mathematics, Zhejiang Normal University, Email: xdzhu@zjnu.edu.cn, Grant numbers: NSFC 11971438,  U20A2068,  ZJNSFC   LD19A010001. }}
	
	\maketitle
	
	
		\begin{abstract}
			For a graph $G$ and a positive integer $n$, the $n$th cone over $G$ is obtained from the direct product $G \times P_n$ of $G$ and a path $P_n=(0,1,\ldots, n)$, by adding a copy of $G$ on $V(G) \times \{0\}$, and  identifying $V(G) \times \{n\}$ into a single vertex $\star$. Assume $G$ and $H$ are graphs, and $h: V(H) \to \mathbb{N}$ is a mapping which assigns to each vertex $v$ of $H$ a positive integer. For each vertex $v$ of $H$, let $\Delta_{h(v)}(G,v)$ be a copy of the $h(v)$-th cone over $G$, with vertex set $V(\Delta_{h(v)}(G)) \times \{v\}$. 
			The $(H,h)$-cone over $G$ is the graph obtained from  the disjoint union of $\{\Delta_{h(v)}(G, v) : v\in V(H)\}$   by identifying $\{((x,0),v): v \in V(H)\}$ into a single vertex $(x,0)$ for each $x \in V(G)$, and adding edges $\{(\star, v) (\star, v'): vv' \in E(H)\}$. When $h(v)=n$ is a constant mapping, then $\Delta_{H,h}(G)$ is denoted by $\Delta_{H,n}(G)$. 
			In this paper, we determines the fractional chromatic number of $\Delta_{H,n}(G)$  for all $G, H$ with $\chi_f(H)\le \chi_f(G)$. 
		\end{abstract}
		
		\section{Introduction}
		
		Graphs in this paper are finite and simple, unless otherwise stated. The {\em direct product} of two graphs $G$ and $H$, denote by $G \times H$, is a graph with vertex set $\{(x,y): x \in V(G), y \in V(H)\}$, in which $(x,y)(x',y')$ is an edge if and only if $xx' \in E(G)$ and $yy' \in E(H)$.	Assume $G$ is a graph and $n$ is  a positive integer. Let $P^{\circ}_n$ be the graph obtained from the path $P_n$ on vertices $\{0, 1, \ldots, n\}$ by adding a loop at vertex $0$.  The  {\em $n$th cone over $G$}, denoted by $\Delta_n(G)$, is obtained from the direct product $G \times P^{\circ}_n$ by identifying $V(G) \times \{n\}$ into a single vertex $\star$. So 
		$$V(\Delta_n(G))= (V(G) \times \{0,1,\ldots, n-1\}) \cup  \{\star\} $$ and 
		$$E(\Delta_n(G)) = \{(x,i)(y,j): xy \in E(G), |i-j|=1 ~\text{or } i=j=0\} \cup \{(x, n-1) \star: x \in V(G)\}.$$
		
		The vertices of $\Delta_n(G)$ are divided into layers: For $i=0,1,\ldots, n-1$, the $i$th layer is $V(G) \times \{i\}$. The $0$th layer $\{(x,0): x \in V(G)\}$   induces a copy of $G$, and is called the {\em base} of $\Delta_n(G)$. Each other layer of $\Delta_n(G)$ is an independent set, adjacent only to the two neighboring layers. The vertex $\star$ is called the {\em apex vertex} of $\Delta_n(G)$, which is adjacent to all vertices in the $(n-1)$th layer.

		Observe that  $\Delta_1(G)$ is obtained from $G$ by adding an universal vertex, and $\Delta_2(G)$ is the well-known Mycielski construction over $G$. It is well-known that for any graph $G$, $\chi(\Delta_2(G)) = \chi(G)+1$ and $\omega(\Delta_2(G)) = \omega(G)$. By iterately applying Mycielski construction to $K_2$, we obtain triangle free graphs of arbitrary large chromatic number. 
		
		The cone over a graph was introduced by Stiebitz \cite{Stiebitz}  as  a generalization of Mycielski construction over a graph. Stiebitz studied the chromatic number of   $\Delta_n(G)$.  It turns out that there are graphs $G$ for which  $\chi(\Delta_n(G)) = \chi(G)$ for all $n \ge 3$.
		However,  the {\em $Z_2$-coindex} of the box complex of  a graph, which induces a lower bound for its chromatic number,  does increase when a cone construction is applied to a graph. This implies that 
		for some graphs $G$ (say  for   complete graphs), each iteration of the cone construction increases the chromatic number by $1$. The cone over graphs can be used to construct graphs of arbitrary large odd girth and arbitrary large chromatic number. 
		
		For a graph $G$, denote by $\mathcal{I}(G)$
		the family of independent sets of $G$. A \emph{fractional colouring} of   $G$ is a 
		mapping $f: \mathcal{I}(G) \to [0,1]$ such that for each vertex $v$, 
		$$\sum_{v \in I, I \in \mathcal{I}(G)} f(I) \ge 1.$$
		The {\em weight} of a fractional colouring $f$ of $G$ is
		$$w(f)= \sum_{I \in \mathcal{I}(G)} f(I).$$
		The {\em fractional chromatic number} $\chi_f(G)$ of $G$ is the minimum weight of a fractional colouring of $G$.

		The definition above shows that the fractional chromatic number of a graph $G$ is the solution of a linear programming problem. The dual of this linear programming problem defines the fractional clique number of $G$. Specifically, 
		a {\em fractional clique} of a graph $G$ is a mapping $\nu:V(G) \to [0,1]$ such that for each independent set $I$ of $G$, $$\nu(I) = \sum_{v \in I} \nu(v) \le 1.$$
		The {\em weight} of a fractional clique $\nu$ of $G$ is $$\nu(V(G)) = \sum_{v \in V(G)} \nu(v).$$ 
		The {\em fractional clique number} $\omega_f(G)$ is the maximum weight of a fractional clique of $G$. 
		It follows from the duality theorem of linear programming that  for any graph $G$, $\omega_f(G) = \chi_f(G)$.

		The fractional chromatic number of Mycielski construction over graphs was studied by Larsen, Propp and Ullman \cite{LPU}, who proved that for any graph $G$, 
		$$\chi_f(\Delta_2(G)) = \chi_f(G)+ \frac {1}{\chi_f(G)}.$$
		This result was generalized by Tardif \cite{Tardif} who proved that for any positive integer $n$, 
		$$\chi_f(\Delta_n(G)) = \chi_f(G)+ \frac {1}{\sum_{k=0}^{n-1} (\chi_f(G)-1)^k}.$$
		Thus by iterately applying cones over $K_2$ and with $n$ large enough, one can construct graphs of arbitrarily large odd girth and arbitrarily large fractional chromatic number.  The existence of graphs of large odd girth and large fractional chromatic number was proved by Erd\H{o}s \cite{erdos} by using probabilistic method.  However, not many such graphs are explicitly constructed and with their fractional chromatic numbers determined. 
		
		Given a real number $r \ge 2$, let $Kn_r$ be the infinite graph whose vertices are measurable subsets of $[0,r]$ of measure $1$, in which two vertices are adjacent if the two subsets are disjoint. A {\em homomorphism} from a graph $G$ to a graph $K$ is a mapping $f:V(G) \to V(K)$ that preserves the edges, i.e., $xy \in E(G)$ implies that $f(x)f(y) \in E(K)$.  We  write $G \to K$ if there is a homomoprhism from $G$ to $K$, and write $G \not\to K$ if such a homomorphism does not exist. 
		It is well-known that a fractional $r$-colouring of a graph $G$ is equivalent to a homomorphism from $G$ to $Kn_r$. 
		
		For (possibly infinite) graphs $G$ and $K$, the {\em exponential graph} $K^G$ has vertices all the mappings $f: V(G) \to V(K)$, in which $f$ and $g$ are adjacent if and only if for any edge $xy$ of $G$, $f(x)g(y)$ is an edge of $K$.  
		In particular, loops in $K^G$ are homomorphisms from $G$ to $K$. 
		
		Exponential graphs play a key role in the study of Hedetniemi's conjecture. We say a (possibly infinite) graph $K$ is {\em multiplicative}  with respect to the family of finite graphs if $G \not\to K$ and $H \not\to K$ implies that $G \times H \not\to K$ for any finite graphs $G$ and $H$. Hedetniemi's conjecture \cite{Hedetniemi} is equivalent to say that complete graphs $K_k$ are multiplicative, and the fractional version of Hedetniemi's conjecture is equivalent to say that for any $r \ge 2$, the inifinite graph $Kn_r$ is multiplicative with respect to the family of finite graphs. It is well-known that a (possibly infinite) graph $K$ is multiplicative
		with respect to the family of finite graphs if and only if for any finite graph $G$, $G \not\to K$ implies that $ H \to K$ for any finite subsgraph $H$ of $K^G$. 
		Shitov \cite{Shitov} refuted Hedetniemi's conjecture by constructing, for sufficiently large $k$, a graph $G$ such that $G \not\to K_k$ and $K_k^G \not\to K_k$. Later on, non-multiplicativity of smaller complete graphs were obtained in a sequence of papers \cite{Tardif2, Tardif3,Wrochna,Zhu}, also by using the concept of exponential graphs. 
		On the other hand, it was shown in \cite{Zhu-fractional} that the fractional Hedetniemi's conjecture holds, namely graphs $Kn_r$ are multiplicative with respect to the family of finite graphs.

		It follows from the definition that 
		a path of length $n$ from a loop-vertex to a constant map in $K^G$ is equivalent to a homomorphism from $\Delta_n(G)$ to $K$: Given a homomorphism $f$ from $\Delta_n(G)$ to $K$, for $i=0,1,\ldots, n$, let $f_i$ be the restriction of $f$ to $V(G) \times \{i\}$. Then 
		$(f_0,f_1, \ldots, f_n)$ is a path in $K^G$ from a loop-vertex $f_0$ to a constant map $f_n$. Note that 
		$V(G) \times \{n\}$ is identified into a single vertex $\star$, and we may think of $f_n$ as the constant map $f_n: V(G) \to V(K)$,   defined as $f_n(v)=f(\star)$ for all $v \in V(G)$.  
		
		Thus the study of the fractional chromatic number of $\Delta_n(G)$ is equivalent to the study of paths of length $n$ in $Kn_r^G$ from loop-vertices to constant maps. It follows from the result of Tardif \cite{Tardif}  that there is a path in $Kn_r^G$ from a loop-vertex to a constant map if and only if $\chi_f(G) < r$, and the length of   a path in $Kn_r^G$ from a loop-vertex to a constant map depends only on $\chi_f(G)$. For a fixed $r$, the smaller is $\chi_f(G)$, the  shorter is such a path.  
		
		It is natural to ask what is the structure of the subgraph of $Kn_r^G$ induced by the constant maps at a fixed distance from a given loop. This partly motivates the following generalization of cones over a graph.

		\begin{definition}
			Assume $H$ is a graph and $h: V(H) \to \mathbb{N}$ is a mapping. The {\em $(H,h)$-cone over $G$} is the graph defined as follows:
			
			For each vertex $v$ of $H$, let $\Delta_{h(v)}(G,v)$ be a copy of $\Delta_{h(v)}(G)$ with vertex set 
			$$V(\Delta_{h(v)}(G)) \times \{v\} =\{(u,v): u \in \Delta_{h(v)}(G)\}.$$
			The $(H,h)$-cone over $G$, denoted by $\Delta_{H,h}(G)$, is   obtained from the disjoint union of $\{\Delta_{h(v)}(G, v) : v \in V(H)\}$ by identifying $\{((x,0),v): v \in V(H)\}$ into a single vertex $(x,0)$ for each $x \in V(G)$, and adding edges $\{(\star,  v) (\star,  v'): vv' \in E(H)\}$.   
		\end{definition}

		If $h(v)=n$ for all $v \in V(H)$, then we denote $\Delta_{H,h}(G)$ by $\Delta_{H,n}(G)$. 
		Thus  $\Delta_n(G) \equiv  \Delta_{K_1,n}(G)$. 
		
		This paper determines the fractional chromatic number of $\Delta_{H,n}(G)$ for all graphs $G,H$ with $\chi_f(H) \le \chi_f(G)$, and for all positive integers $n$. It turns out that $\chi_f(\Delta_{H,n}(G))$ can be expressed as a function of   $\chi_f(G), \chi_f(H)$ and $ n$. This might be useful in some graph construction. Indeed, an early motivation of this construction was to find smaller counterexamples to Hedetniemi's conjecture. Although such graphs can be used in the construction of counterexamples to Hedetniemi's conjecture, much  smaller counterexamples  are found by other means. 
		
		Nevertheless,  the graph operation $\Delta_{H,n}(G)$ is a natural generalization of the cone construction over a graph. We believe they are  of independent interests. 
		In particular, the following observation, which follows from the definition, show that the fractional chromatic number of such graphs are related to the structure of exponential graphs $Kn_r^G$. 
		
		\begin{observation}
			\label{obs-a} 
			For $h: V(H) \to \mathbb{N}$,	$\chi_f(\Delta_{H,h}(G)) \le r$ if and only if there are constant maps $\{\phi_v: v \in V(H)\}$ from $G$ to $Kn_r$ such that the induced subgraph of $Kn_r^G$ contains $H$ as a spanning subgraph,  and these constant maps  are connected to a loop-vertex in $Kn_r^G$ with paths of the corresponding lengths $h(v)$. 
		\end{observation}

	\section{The main results}
	
	Assume $G$ and $H$ are graphs and $n$ is a positive integer. 
	Let 
	\[
	\tau(G,n)=\frac{1}{\sum_{k=0}^{n-1}(\chi_f(G)-1)^k},
	\]
	\[
	\tau'(G,n, H) =\frac{1}{\chi_f(H)(\sum_{k=0}^{n-1}(\chi_f(G)-1)^k)+1-\chi_f(H)}.
	\]

	\begin{theorem}
		\label{thm-main}
		Assume $H$ is a graph with $\chi_f(H) \le \chi_f(G)$.     Then 
		\[
		\chi_f(\Delta_{H,n}(G))= \begin{cases} 
		\chi_f(G) + \tau(G,n) &\text{ if $n $ is even}, \cr
		\chi_f(G)+ \chi_f(H) \tau'(G,n, H) &\text{ if $n $ is odd}.
		\end{cases}
		\]
	\end{theorem}

	The remainder of this paper is devoted to the proof of Theorem \ref{thm-main}. 
	If $n=1$, then $\Delta_{H,n}(G)$ is the join of $G$ and $H$, and hence   $\chi_f(\Delta_{H,1}(G)) = \chi_f(G) +\chi_f(H)$ and  Theorem \ref{thm-main} is true. For the remainder of this paper, we assume that $n \ge 2$. 
	
	The set $\{(x,0): x \in V(G)\}$ induces a copy of $G$, and is called the base of $\Delta_{H,h}(G)$. 
	Note that the base of $\Delta_{H,h}(G)$ is the identification of the bases of $\Delta_{h(v)}(G, v)$. For convenience, we shall view $\{(x,0): x \in V(G)\}$ also as the base of $\Delta_{h(v)}(G, v)$. In other words, we treat $\Delta_{h(v)}(G, v)$ as a subgraph of 
	$\Delta_{H,h}(G)$, and when we discuss about this subgraph, we treat the set  $\{(x,0): x \in V(G)\}$ the same as $\{((x,0), v): x \in V(G)\}$.

	The following observation will be frequently used.
	
	\begin{observation}
		\label{obs-ind}
		A subset $I$ of $\Delta_{H,h}(G)$ is independent if and only if 
		\begin{enumerate}
			\item the restriction of $I$ to $V(\Delta_{h(v)}(G,v))$ is an independent set of $\Delta_{h(v)}(G,v)$ for each vertex $v$ of $H$, and 
			\item the set $\{v \in V(H): (\star, v) \in I\}$ is an independent set of $H$.
		\end{enumerate}  
	\end{observation}

	\section{Fractional cliques in $\Delta_{H,n}(G)$}
	If $n $ is even, then since $\Delta_n(G)$ is a subgraph of $\Delta_{H,n}(G)$, we have    $\chi_f(\Delta_{H,n}(G)) \ge \chi_f(\Delta_n(G)) = \chi_f(G)+\tau(G,n)$. Therefore for $n$ even, it only remains to show that $\chi_f(\Delta_{H,n}(G))\le \chi_f(G)+\tau(G,n)$, which is done in the next section.
	In the remainder of this section, we prove that if    $n \ge 3$ is odd, then  $$\chi_f(\Delta_{H,n}(G)) \ge \chi_f(G) + \chi_f(H)\tau'(G,n, H).$$ 
	For this purpose, it suffices to  construct a fractional clique $\nu'$ of $\Delta_{H,n}(G)$ of  weight $\chi_f(G) + \chi_f(H)\tau'(G,n,H)$.

	For simplicity, let 
	$$\tau'= \tau'(G,n,H) =\frac{1}{\chi_f(H) (\sum_{k=0}^{n-1}(\chi_f(G)-1)^k)+1-\chi_f(H)}.$$

	Let $$\nu: V(G) \to [0,1] \text{ and } \eta: V(H) \to [0,1]$$  be  fractional cliques of $G$ and $H$  respectively, with maximum weights.  We may assume that $G$ is {\em critical}, i.e., every proper subgraph of $G$ has smaller fractional chromatic number. Hence   $\nu(x)>0$ for each vertex $x$ of $G$.
	We shall construct a fractional clique $\nu'$ of $\Delta_{H,n}(G)$. 
	
	First we define a weight function $\theta$  of $\Delta_n(G)$ as follows:
	
	We set $$\theta(\star)=\tau'.$$
	For $i=0,1,\ldots, n-1$,  let $$\alpha_i =   \tau' (\chi_f(G)-1)^{n-1-i}.$$ 
	For $1 \le i \le n-1$, 
	\begin{equation*}
	\theta(x, i)= \alpha_i \nu(x), 
	\end{equation*}
	and let 
	\begin{equation*}
	\theta(x, 0 )= 
	\left(  \alpha_0- \left(1-\frac{1}{\chi_f(H)}\right)\alpha_{n-1}\right) \nu(x).
	\end{equation*}

	For each vertex $v$ of $H$, let $\theta_v: \Delta_n(G,v) \to [0,1]$ be defined as 
	$$\theta_v(x,v)= \theta(x) \eta(v).$$

	Note that 
	\begin{equation}
	\label{eqn-sum}
 \alpha_0+\alpha_1+\alpha_2+ \ldots +   \alpha_{n-2}+\frac{1}{\chi_f(H)}\alpha_{n-1}   = \frac{1}{\chi_f(H)}.
	\end{equation}    	Hence
	$$\theta_v(\Delta_{n}(G,v) - \{(\star,v)\}) = \eta(v)\frac{\chi_f(G)}{\chi_f(H)},$$
	and 
	$$\theta_v(\Delta_{n}(G,v)  ) =\eta(v) \left(\frac{\chi_f(G)}{\chi_f(H)} + \tau'\right).$$
	
	Let $\nu': V(\Delta_{H,n}(G)) \to [0,1]$ be defined as
	\[
	\nu'(z)= \begin{cases} \sum_{v \in V(H)} \theta_v((x,0),v) &\text{ if $z=(x,0)$},\cr
	\theta_v(z) &\text{ if $z \in V(\Delta_n(G,v)) - \{((x,0),v): x \in V(G)\}$}.
	\end{cases}
	\]

	Then 
	
	\begin{eqnarray*}
		\sum_{z \in V(\Delta_{H,n}(G))} \nu'(z) &=& \sum_{v \in V(H)} \theta_v(V(\Delta_n(G,v)) ) \\
		&=&   \sum_{v \in V(H)} \eta(v) \left(\frac{\chi_f(G)}{\chi_f(H)} + \tau'\right) \\ 
		&=& \chi_f(G)  + \tau' \chi_f(H).
	\end{eqnarray*} 
	
	We shall show that $\nu'$ is a fractional clique of $\Delta_{H,n}(G)$. 
	Since $\alpha_i >0$ for each $i$ and  $\alpha_0\ge \alpha_{n-1}$, we know that    $\nu'(z) \ge 0$ for $z \in V(\Delta_{H,n}(G))$.
	It remains to show the following:
	
	\bigskip
	(*) {\em 	 For each independent set $I$ of $\Delta_{H,n}(G)$, $\nu'(I) \le 1$.}
	\bigskip

	Given an independent $J$ of $\Delta_{n}(G)$,  
	$i \in \{0, 1,\ldots,   n-1\}$, let 
	$$J(i) = \{x \in V(G): (x,i) \in J\},$$
	and 
	\[
	J(n) = \begin{cases} \{\star\} & \text{if $\star  \in J$}, \cr
	\emptyset &\text{ otherwise.} \cr
	\end{cases}
	\] 
	Let
	\begin{equation*}
	\beta(J)=\left\{
	\begin{array}{rcl}
	\sum_{i=0}^{n-1}\alpha_i-(1-\frac{1}{\chi_f(H)})\alpha_{n-1}\nu(J(0))   & & \text{ if $\star \in J$, } \\
	\sum_{i=0}^{n-2}\alpha_i+\frac{1}{\chi_f(H)}\alpha_{n-1}\nu(J(0))   & & \text{ if  $\star  \notin J$. } \\
	\end{array} \right.
	\end{equation*}
	\begin{lemma}\label{weight1}
		For any independent set $J$ of $\Delta_n(G)$,
		$\theta(J) = \sum_{u \in J} \theta(u) \le \beta(J).$
	\end{lemma}
	
	Assume Lemma \ref{weight1} holds and $I$ is an independent set of $\Delta_{H,n}(G)$. Let 
	$$S=\{v \in V(H): (\star, v) \in I\}.$$
	For any vertex $v$ of $H$, let $I(v)$ be the restriction of $I$ to $ \Delta_{n}(G,v)$. By Observation \ref{obs-ind}, $I(v)$ is an independent set of $\Delta_{n}(G,v)$ and $S$ is an independent set of $H$. Note that $I(v)(0) = I(0)$ for all $v \in V(H)$. We  have 
	\begin{eqnarray*}
		\nu'(I) &= &\sum_{v\in V(H)}\theta_v(I(v)) 
		\le \sum_{v\in V(H)}\beta(I(v))\eta(v) \\
		&=&\sum_{v\in S} \left(	\sum_{i=0}^{n-1}\alpha_i-(1-\frac{1}{\chi_f(H)})\alpha_{n-1}\nu(I(0)) \right)  \eta (v) + \sum_{v\in V(H)-S} \left(	\sum_{i=0}^{n-2}\alpha_i+\frac{1}{\chi_f(H)}\alpha_{n-1}\nu(I(0))  \right) \eta (v)\\
		&=&\chi_f(H)(\alpha_0+\ldots+\alpha_{n-2})\\
		&+& \sum_{v\in S} \left(\alpha_{n-1}-(1-\frac{1}{\chi_f(H)})\alpha_{n-1}\nu(I(0)) \right)  \eta (v)+  \sum_{v\in V(H )-S} \left(\frac{1}{\chi_f(H)}\alpha_{n-1}\nu(I(0))  \right) \eta (v)\\
		&=&\chi_f(H)(\alpha_0+\ldots+\alpha_{n-2})\\
		&+& \sum_{v\in S} \left(  \alpha_{n-1}-\alpha_{n-1}\nu(I(0))   \right)  \eta (v)+\sum_{v\in V(H)} \left(\frac{1}{\chi_f(H)}\alpha_{n-1}\nu(I(0))  \right) \eta(v).	
	\end{eqnarray*}
	
	As $\sum_{v \in S} \eta(v) \le 1$, $\nu(I(0))\le 1$ and  $\sum_{v \in V(H)} \eta(v)=\chi_f(H)$,
	$$ \sum_{v\in S} \left(\alpha_{n-1}-\alpha_{n-1}\nu(I(0)) \right)  \eta (v) +  \sum_{v\in V(H)} \left(\frac{1}{\chi_f(H)}\alpha_{n-1}\nu(I(0))  \right) \eta(v) \le \alpha_{n-1}.$$
	Hence by (\ref{eqn-simga}), $$	\nu'(I) \le \chi_f(H)(\alpha_0+\ldots+\alpha_{n-2}) + \alpha_{n-1} = 1.$$

	Thus   to prove Statement (*), it suffices to prove  Lemma \ref{weight1}. 
	\bigskip
	\noindent

	{\bf Proof of Lemma \ref{weight1}}	We shall frequently use the following equality, which follows directly from the definition.
	For   $0 \le  k \le  n-2$, 
	\begin{equation}
	\label{eqn-alpha} 
	\alpha_k+\alpha_{k+1} = \alpha_{k+1} \chi_f(G).  
	\end{equation}  
	We shall also frequently use the following observation.
	\begin{observation}
		\label{obs-1}
		If $\nu$ is a weight assignment  to the vertices of a graph $Q$ and $\alpha$ is the maximum weight of an independent set, then
		$\chi_f(Q) \ge \frac{\nu(V(Q))}{\alpha}$. 
	\end{observation} 
	For each $J\in \mathcal{I}(\Delta_{n}(G))$, the {\em level}  of $J$ is the maximal integer $i$ such that 
	$$J(0)=J(2)=\ldots=J(2\lfloor \frac{i}{2} \rfloor) $$ and 
	$$J(1)=J(3)=\ldots=J(2\lfloor \frac{i-1}{2} \rfloor+1)=V(G)-N(J(0)).$$
	Assume Lemma \ref{weight1} is not true. Let  
	$J\in \mathcal{I}(\Delta_{n}(G))$ be an independent set such that 
	\begin{itemize}
		\item[(1)] $\theta(J)- \beta(J)>0$ is maximum.
		\item[(2)]  Subject to $(1)$, $\nu(J(0))+\nu(J(1))$ is maximum.
		\item[(3)] Subject to $(1)$  and $(2)$, the level $i_0$ of $J$ is maximum.
	\end{itemize}
	\begin{claim}
		\label{clm-1}
		$i_0\ge 1$.
	\end{claim}
	\begin{proof}
		Let $A=J(0)-J(1)$, $B=J(2)-J(1)-J(0)$ and $H=G[A\cup B]$. 
		
		First we show that $A$ is an independent set of $H$ with maximum weight.
		Assume to the contrary that $N\in \mathcal{I}(H)$ is an independent set such that $\nu(N)>\nu(A)$. Let
		$$\delta = \nu(N)-\nu(A),  \ J'=(J-A\times \{0\} )\cup (N\times \{0\}).$$
		Since for any vertex $u$ in $A\cup B$ and for any vertex $v$ in $J(1)$, $uv\notin E(G)$, $J'$ is an independent set of $\Delta_n(G)$. As $\alpha_0-\alpha_{n-1}\ge 0$,  we have
		\begin{eqnarray*}
			\theta(J')- \beta(J')&\ge& \theta(J)+(\alpha_0-(1-\frac{1}{\chi_f(H)})\alpha_{n-1})\delta   -(\beta(J)+\frac{1}{\chi_f(H)}\alpha_{n-1}\delta ) \\
			&=& \theta(J)-\beta(J)+(\alpha_0-\alpha_{n-1})\delta  \\
			&\ge&\theta(J)-\beta(J).
		\end{eqnarray*}
		
		Since $\nu(J'(0))+\nu(J'(1))> \nu(J(0))+\nu(J(1))$, this is in contrary to the choice of $J$.
		
		If $A\neq \emptyset,$  then  $\nu(A) > 0$,  and by Observation \ref{obs-1} and Claim \ref{clm-1},    $\chi_f(G)\ge \chi_f(H)\ge \frac{\nu(A)+\nu(B)}{\nu(A)}$. Since $\alpha_{1}+\alpha_{2}=\alpha_{2}\chi_f(G)$, we have
		$ \alpha_1\nu(A)\ge \alpha_{2}\nu(B).$
		Let $$J'=(J-B\times \{2\} )\cup (A\times \{1\}).$$ Since there is no edge between $A$ and $J(0)\cup J(1)$, and $J(2)-B\subseteq J(0)\cup J(1)$,   we know that $J'$ is an independent set of $ \Delta_{n}(G)$ with $\theta(J')- \beta(J')\ge \theta(J)- \beta(J)$ and $\nu(J'(0))+\nu(J'(1))> \nu(J(0))+\nu(J(1))$, a contradiction.  Hence, $A=\emptyset$ which implies that $B=\emptyset$, and hence $J(0) \subseteq J(1)$ and $J(2) \subseteq J(1)$.

		Let $C=J(2)-J(0)$. If $C \ne \emptyset$, then let
		$$J'=J\cup (C \times\{0\}).$$
		Since $ J(2)\cap J(1)$ is an independent set and $C \subseteq J(2) \cap J(1)$, $J'$ is an independent set. We have
		\begin{eqnarray*}
			\theta(J')- \beta(J')&\ge& \theta(J)+(\alpha_0-(1-\frac{1}{\chi_f(H)})\alpha_{n-1})\nu(C)  -(\beta(J)+\frac{1}{\chi_f(H)}\alpha_{n-1}\nu(C) )\\
			&=& \theta(J)-\beta(J)+(\alpha_0-\alpha_{n-1})\nu(C) \\
			&\ge&\theta(J)-\beta(J).
		\end{eqnarray*}
		Since $\nu(J'(0))+\nu(J'(1))> \nu(J(0))+\nu(J(1))$, this is contrary to the choice of $J$. Hence  $A=B=C=\emptyset$ and $J(2)\subseteq J(0) \subseteq J(1).$
		Since $J$ is an independent set of $\Delta_n(G)$, we know that $J(1) \subseteq V(G)-N(J(0))$. Since $J(2) \subseteq J(0)$, $J \cup (V(G)-N(J(0))) \times \{1\}$ is an independent set of $\Delta_n(G)$. By the maximality of $J$, 
		we have $J(1)=V(G)-N(J(0))$, and hence $i_0\ge 1$. 
	\end{proof}
	\begin{claim}
		\label{claim}   $i_0\ge n -2$.  
	\end{claim}
	\begin{proof}
		Assume $i_0 \le n-3$. 
		If $i_0$ is even, then $J(i_0)=J(0)$ and $J(i_0+1)\subseteq V(G)-N(J(0))=J(i_0-1)$. If $i_0$ is odd, then $ J(i_0)=J(1)=V(G)-N(J(0))$ and $J(i_0+1)\subseteq J(0)=J(i_0-1)$. So in any case, $J(i_0+1)\subseteq J(i_0-1)$. 
		By the maximality of $i_0$, we know that $A=J(i_0-1)-J(i_0+1)\neq \emptyset$. Hence $B=J(i_0+2)-J(i_0)\neq \emptyset$ (otherwise, by replacing
		$J(i_0+1) \times \{i_0+1\}$ with $J(i_0-1) \times \{i_0+1\}$, 
		we  obtain an independent set $J'$ with $\theta(J')- \beta(J')>\theta(J)- \beta(J)$).
		Let $$J'=(J-B\times\{i_0+2\})\cup (A\times\{i_0+1\}).$$ We have
		$$  \theta(J')- \beta(J')\ge \theta(J)+\alpha_{i_0+1}\nu(A)-\alpha_{i_0+2}\nu(B)-\beta(J).$$
		If $\alpha_{i_0+1}\nu(A)\ge \alpha_{i_0+2}\nu(B)$, then $J'$ is an independent set with $\theta(J')-\beta(J')\ge \theta(J)-\beta(J)$ and $\nu(J'(0))+\nu(J'(1))=\nu(J(0))+\nu(J(1))$ but the level of $J'$ is greater than $i_0$. This is in contrary to the choice of $J$.
		Hence  $\alpha_{i_0+2}\nu (B)> \alpha_{i_0+1}\nu (A).$ By Equality (\ref{eqn-alpha}), for $k=0,1,\ldots, n-2$,
		$$\alpha_{k+1}\nu (B)> \alpha_{k}\nu (A).$$ 
		Let $$J'=(J-A\times\{i_0-1\})\cup (B\times\{i_0\}).$$ Then $J'$ is an independent set. If $i_0=1$, then
		\begin{eqnarray*}
			\theta(J')- \beta(J')&\ge & \theta(J)+\alpha_{1}\nu(B)\\
			&-&(\alpha_0-(1-\frac{1}{\chi_f(H)})\alpha_{n-1})\nu(A)  -(\beta(J)+\frac{1}{\chi_f(H)}\alpha_{n-1}\nu(A) )\\
			&=&\theta(J)-\beta(J)+\alpha_{1}\nu(B) -\alpha_0\nu(A)+(1-\frac{2}{\chi_f(H)})\alpha_{n-1}\nu(A) \\
			&\ge&\theta(J)-\beta(J)+\alpha_{1}\nu(B) -\alpha_0\nu(A)  \text{\ (as $\chi_f(H)\ge 2$)}\\
			&>&\theta(J)-\beta(J)
		\end{eqnarray*}
		If $i_0\ge 2$, then $
		\theta(J')- \beta(J')= \theta(J)+\alpha_{i_0}\nu(B)-\alpha_{i_0-1}\nu(A)-\beta(J)>\theta(J)-\beta(J)$,   a contradiction.
		This completes the proof of Claim \ref{claim}.
	\end{proof}
	Since $J(n-2)=J(1)=V(G)-N(J(0))$, we know that $J(n-1)\subseteq J(0).$ If $(\star,v)\in J$, then
	\begin{eqnarray*}
		J=\left(J(0)\times\{0,2,\ldots,n-3\}\right) \cup \left((V(G)-N(J(0)))\times\{1,3,\ldots,n-2\} \right)\cup(\star,n).
	\end{eqnarray*}
	If $(\star,v)\notin J$, then
	\begin{eqnarray*}
		J=\left(J(0)\times\{0,2,\ldots,n-1\}\right) \cup \left((V(G)-N(J(0)))\times\{1,3,\ldots,n-2\} \right).
	\end{eqnarray*}
	Assume $A=J(0)$ and $B=J(1)-J(0)$. Let $H=G[B]$ and $N\in \mathcal{I}(H)$ be an independent set of $H$ with maximal weight. Then $$\chi_f(G)\ge \chi_f(H)\ge \frac{\nu(B)}{\nu(N)}.$$
	Since $A\cup N$ is an independent set of $G$ and $A\cap N=\emptyset$ which implies that $\nu(A)+\nu(N)\le 1$, we have
	\begin{eqnarray*}
		\alpha_{k+1}\nu(N(J(0)))-\alpha_k\nu(J(0))&=& 
		\alpha_{k+1}\nu(V(G)-A-B)-\alpha_k\nu(A)\\
		&=&\alpha_k+\alpha_{k+1}-(\alpha_k+\alpha_{k+1})\nu(A)-\alpha_{k+1}\nu(B)\\
		&\ge&\alpha_k+\alpha_{k+1}-(\alpha_k+\alpha_{k+1})\nu(A)-\alpha_{k+1}\chi_f(G)\nu(N)\\
		&=&\alpha_k+\alpha_{k+1}-(\alpha_k+\alpha_{k+1})(\nu(A)+\nu(N))\\
		&\ge& 0.
	\end{eqnarray*}
	Hence, for any $0\le k\le n-2$,
	$$\alpha_{k+1}\nu(N(J(0)))\ge \alpha_k\nu(J(0)).$$
	If $\star \in J$, then 
	\begin{eqnarray*}
		\theta_v(J)
		&=&\left(  \alpha_0+\alpha_2+\ldots+\alpha_{n-3}-(1-\frac{1}{\chi_f(H)})\alpha_{n-1}  \right)  \nu(J(0))\\
		&+&(\alpha_1+\alpha_3+\ldots+\alpha_{n-2})\nu(V(G)-N(J(0)))+\alpha_{n-1}\\
		&\le& (\alpha_1+\alpha_3+\ldots+\alpha_{n-2})\nu(V(G))+\alpha_{n-1}-(1-\frac{1}{\chi_f(H)})\alpha_{n-1}\nu(J(0))\\
		&=&(\alpha_0+\alpha_1+\ldots+\alpha_{n-1})-(1-\frac{1}{\chi_f(H)})\alpha_{n-1}\nu(J(0)).
	\end{eqnarray*}
	If $\star \notin J$, then 
	\begin{eqnarray*}
		\theta_v(J)
		&=&(\alpha_0+\alpha_2+\ldots+\alpha_{n-1})\nu(J(0))\\
		&+&(\alpha_1+\alpha_3+\ldots+\alpha_{n-2})\nu(V(G)-N(J(0)))-(1-\frac{1}{\chi_f(H)})\alpha_{n-1}\nu(J(0))\\
		&\le& (\alpha_1+\alpha_3+\ldots+\alpha_{n-2})\nu(V(G))+\frac{1}{\chi_f(H)}\alpha_{n-1}\nu(J(0))\\
		&=&(\alpha_0+\alpha_1+\ldots+\alpha_{n-2})+\frac{1}{\chi_f(H)}\alpha_{n-1}\nu(J(0)).
	\end{eqnarray*}	
	This completes the proof of Lemma \ref{weight1}, and hence
	$\chi_f(\Delta_{H,n}(G)) \ge \chi_f(G) + \chi_f(H) \tau'(G,n,H).$ 

	\section{Fractional colouring}
	
	\subsection{$n$ is even}
	
	This section proves that if $n$ is even, and $\chi_f(H) \le \chi_f(G)$, then 
	$$\chi_f(\Delta_{H,n}(G)) = \chi_f(\Delta_n(G)) = \chi_f(G) + \tau,$$
	where $\tau=\tau(G,n)=\frac{1}{\sum_{k=0}^{n-1}(\chi_f(G)-1)^k}$.
	We already know that 
	$\chi_f(\Delta_{H,n}(G)) \ge \chi_f(\Delta_n(G)) = \chi_f(G) + \tau(G,n).$ 
	It remains  to prove that  $\chi_f(\Delta_{H,n}(G)) \le  \chi_f(G) + \tau(G,n).$ 
	
	For positive integers $s,t$, let $K(s,t)$ be the {\em Kneser} graph, whose vertices are $t$-subsets of $[s]=\{1,2,\ldots, s\}$, and two vertices $A,B$ are adjacent if $A \cap B = \emptyset$ (as subsets of $[s]$). It is well-known (and easy to see) that $\chi_f(H) \le s/t$ if and only if $H \to K(sm,tm)$ for some integer $m$.
	
	It is obvious that if $H \to H'$, then
	$ \Delta_{H,n}(G) \to  \Delta_{H', n}(G)$. 
	Hence 
	$\chi_f(\Delta_{H,n}(G))  \le \chi_f(\Delta_{H', n}(G))$.
	Assume $\chi_f(G) = \frac st$. As $\chi_f(H) \le \frac st$, we may assume that $H \to K(s,t)$ (with appropriate choice of $s$ and $t$, note that $s$ and $t$ need not be coprime). 
	Therefore, 
	it suffices to show that for $H=K(s,t)$,
	$ \Delta_{H,n}(G)$ has a fractional  colouring $\mu': \mathcal{I}(\Delta_{H,n}(G)) \to [0,1]$ of weight $  \chi_f(G) + \tau$.

	For $j \in [s]$, let $$T_j = \{v \in V(H): j \in v\}, \text{and } \overline{T_j} = V(H)-T_j.$$	
	Assume  $I\in \mathcal{I}(G)$ and $ j \in [s]$. If $k \in \{0,2,4, \ldots, n-2\}$ is even,  then let 
	\begin{eqnarray*}
		I_{k,j} &=& \left(I\times\{0,1,\ldots,k\} \times T_j \right)
		\cup \left(V(G)\times\{k+2,k+4,\ldots,n-2\} \times T_j \right)\cup (\{\star\} \times T_j) \\
		&\cup& \left(I\times\{0,1,\ldots,k+1\} \times \overline{T_j}\right)\cup \left(V(G)\times\{k+3,k+5,\ldots,n-1\} \times \overline{T_j} \right).
	\end{eqnarray*}
	
	Let 
	$$O=(V(G)\times\{ 1, 3,\ldots, (n-1)\}) \times V(H).$$

	For $i=0,2,\ldots, n-2$, let
	$$\sigma_i =\frac 1t \tau(\chi_f(G)-1)^i.$$
	Note that for $0\le k\le n-2$, $\sigma_k+\sigma_{k+1}=\sigma_k\chi_f(G)$.	Therefore
	
	\begin{equation}
	\label{eqn-simga}\sigma_0+\sigma_2+\ldots+\sigma_{n-2}= (\sum_{i=0}^{n-1} \sigma_i) /\chi_f(G)  = \frac{1}{s}.
	\end{equation}

	Let $\mu:\mathcal{I}(G)\longrightarrow [0,1]$ be a fractional colouring of $G$ of weight $\chi_f(G)$.
	
	We define 
	$\mu':\mathcal{I}(\Delta_{H,n}(G))\longrightarrow [0,1]$ by 
	\begin{equation*}
	\mu'(J)=\left\{
	\begin{array}{rcl}
	\sigma_k\mu(I)& & \text{if $J=I_{k,j}, I\in \mathcal{I}(G), j \in [s], 0\leqslant k\equiv 0 \pmod{2}\leqslant n-2$}, \\
	\tau& & \text{if $J=O$},\\
	0 & & \text{otherwise}.
	\end{array} \right.
	\end{equation*}

	\begin{figure}[ht]	
		\centering
		\includegraphics[width=6in]{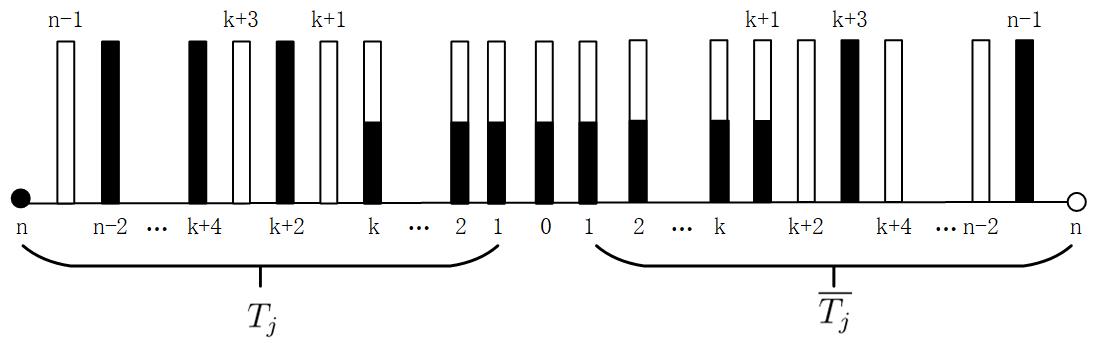}
		\caption{The independent set $I_{k,j}$ with weight $\sigma_k\mu(I)$. In this figure as well as the later figures, if the $i$th column is a filled rectangle, then the whole layer $V(G) \times \{i\}$ is   contained in the independent set $I_{k,j}$, if the $i$th column is a half-filled rectangle, then   $I \times \{i\}$ is   contained in the independent set $I_{k,j}$. Otherwise the layer is disjoint from $I_{k,j}$. The lefthand side represents $I_{k,j} \cap \Delta_n(G,v)$ for  each vertex $v \in T_j$, and the righthand side represents $I_{k,j} \cap \Delta_n(G,v)$ for  each vertex $v \in \overline{T_j}$. }
		\label{fig1}
	\end{figure}
	\begin{figure}[ht]
		
		\centering
		\includegraphics[width=4in]{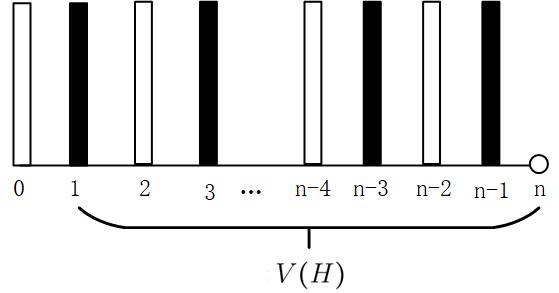}
		\caption{the independent set $O$ with weight $\tau$.}
		\label{fig2}
	\end{figure}
	
	We shall show that $\mu'$ is a fractional colouring of $\Delta_{H,n}(G)$ of weight $\chi_f(G)+ \tau$.
	
	First we calculate the total weight
	$$\sum_{J \in \mathcal{I}(\Delta_{H,n}(G))}\mu'(J) = \sum_{I \in \mathcal{I}(\Delta_n(G)), 0 \le k\equiv 0\pmod{2}\le n-2, j \in [s]}\mu'(I_{k,j})  + \mu'(O).$$
	For   $I \in \mathcal{I}(G)$, it follows from (\ref{eqn-simga}) that 
	$$
	\sum_{0 \le k\equiv 0\pmod{2}\le n-2, j \in [s]}\mu'(I_{k,j}) = s \times \mu(I) \times \sum_{0 \le k\equiv 0\pmod{2}\le n-2 } \sigma_k = \mu(I).$$
	By definition,  $\mu'(O)= \tau$,	
	we have	
	$$\sum_{J\in \mathcal{I}(\Delta_{H,n}(G))}\mu'(J)=\sum_{I \in \mathcal{I}(G)} \mu(I) + \tau = \chi_f(G)+\tau.$$

	It  remains to verify that $\sum_{u\in J}\mu'(J)\ge 1$ for every vertex $u$ of $\Delta_{H,n}(G)$.

	For $u \in V(\Delta_{H,n}(G))$, let $$K(u) = \{(k,j): u \in I_{k,j}\}.$$
	
	\bigskip
	\noindent 	
	{\bf Case 1} $u=(\star, v)$.\\
	
	In this case,  $$K(u) = \{(k,j): 0 \le k \le n-2, k \equiv 0 \pmod{2}, j \in v\}.$$  
	Hence
	\begin{eqnarray*}
		\sum_{u\in J}\mu'(J)&=&\sum_{I\in \mathcal{I}(G), j \in v}t \times [(\sigma_0+\sigma_2+\ldots+\sigma_{n-2})]\mu(I)\\
		&=&t \times[(\sigma_0+\sigma_2+\ldots+\sigma_{n-2})]\chi_f(G)\\
		&=&\frac ts\chi_f(G)=1.\\
	\end{eqnarray*}
	\bigskip
	\noindent
	{\bf Case 2} $u=((x,0),v)$. \\
	Then $$K(u) = \{ (k,j): 0 \le k \le n-2, k \equiv 0 \pmod{2},x \in I , j \in [s] \}.$$ 
	Therefore
	\begin{eqnarray*}
		\sum_{u\in J}\mu'(J)&=&\sum_{x\in I, j \in [s] , I\in \mathcal{I}(G)}[\sigma_0+\sigma_2+\ldots+\sigma_{n-2}]\mu(I)\\
		&=&s(\sigma_0+\sigma_2+\ldots+\sigma_{n-2})=1.
	\end{eqnarray*}
	
	\bigskip
	\noindent
	{\bf Case 3} $u=((x, i),v)$, $i > 0$.\\
	\begin{enumerate}
		\item If $i$ is even, then 
		\begin{eqnarray*}
			K(u) &=& \{(k,j): 0\le k \le i-2, k \equiv 0 \pmod{2}, j \in v\} \\
			&\bigcup& \{ (k,j): k=i, i+2, \ldots, n-2,  x \in I, j \in [s]\}. 
		\end{eqnarray*}
		Hence
		\begin{eqnarray*}
			\sum_{u\in J}\mu'(J)&=&\sum_{I\in \mathcal{I}(G), j \in v}(\sigma_0+\sigma_2+\ldots+\sigma_{i-2})\mu(I)\\
			&+&\sum_{x \in I, I\in \mathcal{I}(G), j \in [s]} (\sigma_i+\sigma_{i+2} + \ldots+\sigma_{n-2})\mu(I)\\
			&=&t \times (\sigma_0+\ldots+\sigma_{i-1})+s \times (\sigma_i+\sigma_{i+2} + \ldots+\sigma_{n-2}) \\
			&=&t\times (\sigma_0+\sigma_2+\ldots+\sigma_{i-2})+t\times (\sigma_1+\sigma_3+\ldots+\sigma_{i-1})+s \times (\sigma_i+\sigma_{i+2} + \ldots+\sigma_{n-2})\\
			&=&t\times (\sigma_0+\sigma_2+\ldots+\sigma_{i-2})+(s-t)\times (\sigma_0+\sigma_2+\ldots+\sigma_{i-2})+s \times (\sigma_i+\sigma_{i+2} + \ldots+\sigma_{n-2})\\
			&=&s\times (\sigma_0+\sigma_2+\ldots+\sigma_{n-2})\\
			&=&1.
		\end{eqnarray*}

		\item If $i$ is odd, then 
		\begin{eqnarray*}
			K(u) &=& \{(k,j): 0 \le k \le i-3, k \equiv 0 \pmod{2}, j \notin v\} \\
			&\bigcup& \{(k,j): k=i-1, j \notin v, x \in I\} \cup  \{ (k,j): k=i+1,\ldots, n-2, x \in I, j \in [s]\}.
		\end{eqnarray*}
	\end{enumerate}
	Moreover,  $u \in O$. 
	Hence 
	\begin{eqnarray*}
		\sum_{u\in J}\mu'(J)&=&\sum_{I\in \mathcal{I}(G), j \notin v}(\sigma_0+\sigma_2+\ldots+\sigma_{i-3})\mu(I)
		+\sum_{x \in I, I \in \mathcal{I}(G), j \notin v}\sigma_{i-1}\mu(I) \\
		&& +\sum_{x\in I, I\in \mathcal{I}(G), j \in [s]}(\sigma_{i+1}+ \sigma_{i+3}+ \ldots+ \sigma_{n-2})\mu(I) 
		+\tau\\
		&=&(s-t) \times (\sigma_0+\sigma_1+\ldots+\sigma_{i-1}) +s \times (\sigma_{i+1}+\sigma_{i+3}+\ldots+\sigma_{n-2}) +\tau\\
		&=&s \times (\sigma_0+\sigma_2+\ldots+\sigma_{n-1})\\
		&=&1,
	\end{eqnarray*}
	where we used the equality
	\begin{eqnarray*}
		t\times (\sigma_0+\sigma_2+\ldots+\sigma_{i-1})&=&t\times \sigma_0+t\times (\sigma_2+\sigma_4+\ldots+\sigma_{i-1})\\
		&=& \tau+ (s-t)\times (\sigma_1+\sigma_3+\ldots+\sigma_{i-2}).
	\end{eqnarray*}
	that is
	\begin{eqnarray*}
		s\times (\sigma_0+\sigma_2+\ldots+\sigma_{i-1})&=&(s-t)\times (\sigma_0+\sigma_2+\ldots+\sigma_{i-1})+\tau+ (s-t)\times (\sigma_1+\sigma_3+\ldots+\sigma_{i-2})\\
		&=&\tau+(s-t)\times (\sigma_0+\sigma_1+\ldots+\sigma_{i-1}).
	\end{eqnarray*}

	This completes the proof of  Theorem \ref{thm-main} for even $n$.

\subsection {$n$ is odd}

 Assume $\chi_f(H) = \frac st \le \chi_f(G)$. Similarly, it suffices to consider the case that $H=K(s,t)$, where $s \ge 2t$.

 Recall that $$	\tau'(G,n,H) =\frac{1}{\chi_f(H)     (\sum_{k=1}^{n-1}(\chi_f(G)-1)^k) +1}.$$
 
 Similarly, for $j\in [s]$, let $T_j=\{v\in V(H): j\in v\}$ and $\overline{T_j}=V(H)-T_j$.
 
 For $I\in \mathcal{I}(G)$, $k\in \{1,3,\ldots,n-2\}$ is odd  and $j\in[s]$, we define $I_{k,j}\in \mathcal{I}(\Delta_{H,n}(G))$ as follows:

 \begin{eqnarray*}
 	I_{k,j} &=&
 	\left(I\times\{1,\ldots,k+1\} \times \overline{T_j} \right)\cup \left(V(G)\times\{k+3,k+5,\ldots,n-1\}\times \overline{T_j} \right)\\
 	&\cup&  \left(I\times\{0,1,\ldots,k\} \times T_j \right)
 	\cup \left(V(G)\times\{k+2,k+4,\ldots,n-2\}  \times T_j \right)  \cup \{(\star,v):  v\in T_j\}. \\
 \end{eqnarray*}

  For $I\in \mathcal{I}(G)$, and $k\in \{0,2,\ldots,n-1\}$ is even, we define $I_{k}\in \mathcal{I}(\Delta_{H,n}(G))$ as follows:
 $$I_{k}=\left(  I\times \{0,1,2\ldots,k\}\times V(H)  \right) \cup \left( V(G)\times \{k+2, k+4, \ldots, n-1\}\times V(H) \right).$$
 For $j\in [s]$, let
 $$O_j=\left(V(G)\times \{1,3,\ldots, n-2\}\times V(H) \right) \cup \{(\star,  v): v\in T_j\}.$$

 For $i=0,1,\ldots, n-1$, let 
 $$\sigma'_i=\frac{\chi_f(H)}{\chi_f(G)}\frac{\tau'}{t} (\chi_f(G)-1)^i.$$
 Let 
 \[
 \delta_i= \begin{cases} s\sigma'_0+ \frac{\sigma'_0(t\chi_f(G)-s)}{\chi_f(G)-2}\left( (\chi_f(G)-1)-1\right), &\text{ if $i=0$},\cr 
 \frac{\sigma'_0(t\chi_f(G)-s)}{\chi_f(G)-2}\left( (\chi_f(G)-1)^{i+1}-(\chi_f(G)-1)^{i-1} \right), &\text{ if $i=2,4,\ldots, n-3$}, \cr
 \frac{\chi_f(G)-\chi_f(H)}{\chi_f(G)}- \frac{\sigma'_0(t\chi_f(G)-s)}{\chi_f(G)-2}\left( (\chi_f(G)-1)^{n-2}-1 \right)- (s-t)\sigma'_0, & \text{ if $i=n-1$}.
 \end{cases}
 \]
 

Note that for $0\le k\le n-2$, $\sigma'_k+\sigma'_{k+1}=\sigma'_k\chi_f(G)$.

\begin{lemma}
	\label{lem-equlaities}
	The following equalities hold:
	\begin{equation}
	\label{eqn-sigma}
s(\sigma'_1+\sigma'_2+\ldots+\sigma'_{n-1})=\frac{\chi_f(H)}{\chi_f(G)}-t\sigma'_0,
	\end{equation}
	\begin{equation}
	\label{eqn-delta}
	\delta_0+\delta_2+\ldots+\delta_{n-3}+\delta_{n-1}=t\sigma'_0+  \frac{\chi_f(G)-\chi_f(H)}{\chi_f(G)}.
	\end{equation} 
	 \begin{equation}
	 \label{eqnsigmadelta}
	  s(\sigma'_1+\sigma'_2+\ldots+\sigma'_{n-1})+\delta_0+\delta_2+\ldots+\delta_{n-3}+\delta_{n-1} = 1.
	 \end{equation}
	 For any even integer $2 \le i \le n-1$, 
	 \begin{equation}
	 \label{eqn-complex}
	 \begin{aligned}
	 &(s-t)(\sigma'_1+\sigma'_2+\ldots + \sigma'_{i-1})\chi_f(G) + (\delta_0+\delta_2+\ldots + \delta_{i-2})\chi_f(G) \\
	 &= s(\sigma'_1+\sigma'_2+\ldots + \sigma'_{i})+(\delta_0+\delta_2+ \ldots + \delta_{i-2}).
	 \end{aligned}
	 \end{equation}
	  For any odd integer $1 \le i \le n-2$, 
	  \begin{equation}
	  \label{eqn8}
	    t (\sigma'_0+\sigma'_1+ \ldots + \sigma'_{i-1})\chi_f(G) - s  ( \sigma'_1+ \ldots + \sigma'_{i-1}) = (\delta_0+\delta_2+\ldots + \delta_{i-1}).
	  \end{equation}
\end{lemma}
\begin{proof}
	(\ref{eqn-sigma}) and (\ref{eqn-delta}) follow directly from the definitions, (\ref{eqnsigmadelta}) follows from 	(\ref{eqn-sigma}) and (\ref{eqn-delta}).
	
Now we prove (\ref{eqn-complex}). It follows from definitions that 	 
	 \begin{equation*}
	 \label{eqn-complexproof1}
	 \begin{aligned}
	 &  (\delta_0+\delta_2+ \ldots + \delta_{i-2})\chi_f(G) -(\delta_0+\delta_2+ \ldots + \delta_{i-2})  \\
	 &=[s\sigma'_0 + \frac{\sigma'_0(t\chi_f(G)-s)}{\chi_f(G)-2}((\chi_f(G)-1)^{i-1} -1)](\chi_f(G)-1)\\
	 &=  s \sigma'_1 +  \frac{\sigma'_0(t\chi_f(G)-s)}{\chi_f(G)-2}((\chi_f(G)-1)^i -(\chi_f(G)-1)), 
	 \end{aligned}
	 \end{equation*}	 
	 and 
	  \begin{equation*}
	  \label{eqn-complexproof2}
	  \begin{aligned}
	  &s (\sigma'_1+\sigma'_2+\ldots + \sigma'_i) -  (s-t)\chi_f(G)   (\sigma'_1+\sigma'_2+\ldots + \sigma'_{i-1})  \\
	  &=[s - (s-t)\chi_f(G)]\sigma'_0 \frac{(\chi_f(G)-1)^i - (\chi_f(G)-1)}{\chi_f(G)-2} + s \sigma'_0 (\chi_f(G)-1)^i  \\
	  &= (t\chi_f(G)-s)   \sigma'_0 \frac{(\chi_f(G)-1)^i - (\chi_f(G)-1)}{\chi_f(G)-2}\\
	  & +(2s-s \chi_f(G) ) \sigma'_0 \frac{(\chi_f(G)-1)^i - (\chi_f(G)-1)}{\chi_f(G)-2} + s \sigma'_0 (\chi_f(G)-1)^i  \\
	  &= \frac{\sigma'_0(t\chi_f(G)-s)}{\chi_f(G)-2}((\chi_f(G)-1)^i -(\chi_f(G)-1)) \\
	  & - s \sigma'_0 [(\chi_f(G)-1)^i - (\chi_f(G)-1) ]+ s \sigma'_0 (\chi_f(G)-1)^i  \\
	   &=  s \sigma'_1 +  \frac{\sigma'_0(t\chi_f(G)-s)}{\chi_f(G)-2}((\chi_f(G)-1)^i -(\chi_f(G)-1)). 
	  \end{aligned}
	  \end{equation*}
	So $$s (\sigma'_1+\sigma'_2+\ldots + \sigma'_i) -  (s-t)\chi_f(G)   (\sigma'_1+\sigma'_2+\ldots + \sigma'_{i-1}) =  (\delta_0+\delta_2+ \ldots + \delta_{i-2})\chi_f(G) -(\delta_0+\delta_2+ \ldots + \delta_{i-2})$$
	and (\ref{eqn-complex}) holds.
	
	Next we prove (\ref{eqn8}).
	 \begin{equation*}
	 \label{eqn8proof}
	 \begin{aligned}
	 & t (\sigma'_0+\sigma'_1+ \ldots + \sigma'_{i-1})\chi_f(G) - s  ( \sigma'_1+ \ldots + \sigma'_{i-1})  \\
	 &=\frac{( t\chi_f(G)-s)\sigma'_0}{\chi_f(G)-2}  ((\chi_f(G)-1)^i - (\chi_f(G)-1)) + t \sigma'_0\chi_f(G)\\
	  &=\frac{( t\chi_f(G)-s)\sigma'_0}{\chi_f(G)-2}  ((\chi_f(G)-1)^i - (\chi_f(G)-1)) + (t\chi_f(G)-s) \sigma'_0 + s\sigma'_0 \\
	  &=\frac{( t\chi_f(G)-s)\sigma'_0}{\chi_f(G)-2}  ((\chi_f(G)-1)^i - (\chi_f(G)-1)+ (\chi_f(G)-2)) + s\sigma'_0 \\
	  &=\frac{( t\chi_f(G)-s)\sigma'_0}{\chi_f(G)-2}  ((\chi_f(G)-1)^i - 1) + s\sigma'_0 \\
	  & = \delta_0+\delta_2+\ldots + \delta_{i-1}.
	 \end{aligned}
	 \end{equation*}
	
\end{proof}

 Let $\mu:\mathcal{I}(G)\longrightarrow [0,1]$ be a fractional colouring of $G$ of weight $\chi_f(G)$.
 
 We define 
 $\mu':\mathcal{I}(\Delta_{H,n}(G))\longrightarrow [0,1]$ by 
 
 \[
 \mu'(J)=
 \begin{cases}
 (\sigma'_{k}+\sigma'_{k+1})\mu(I) &  \text{if $J\in I_{k,j}, I\in \mathcal{I}(G), 1 \leq k \equiv1\pmod{2}\leq n-2, j\in[s]$},\cr 
 \delta_k\mu(I)& \text{if $J=I_{k},  I\in \mathcal{I}(G), 0 \le k\equiv 0\pmod{2}\le n-1$},\cr
 \frac{\tau'}{t} &  \text{if $J=O_j, j\in[s]$},\cr
 0 &  \text{otherwise}.\cr
 \end{cases}
 \]

 \begin{figure}[h]
 	\centering
 	\includegraphics[width=6.5in]{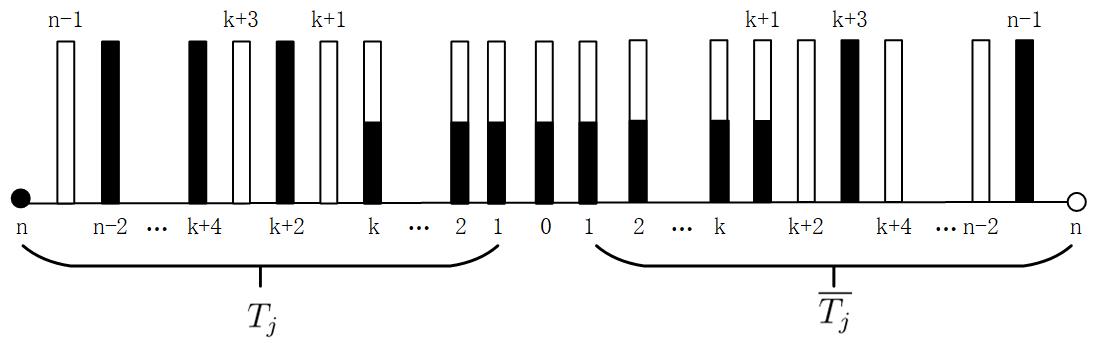}
 	\caption{the independent set $I_{k,j}$ with weight $(\sigma'_k+\sigma'_{k+1})\mu(I)$.}
 	\label{fig3}
 \end{figure}
 
 \begin{figure}[h]
 	\centering
 	\includegraphics[width=5in]{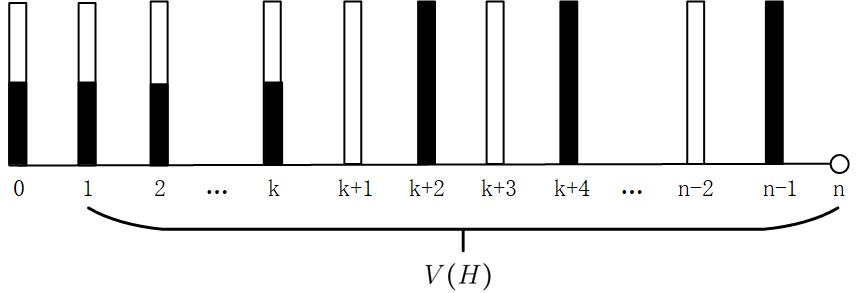}
 	\caption{the independent set $I_k$ with weight $\delta_k\mu(I)$.}
 	\label{fig4}
 \end{figure}
 
 \begin{figure}[h]
 	\centering
 	\includegraphics[width=6.5in]{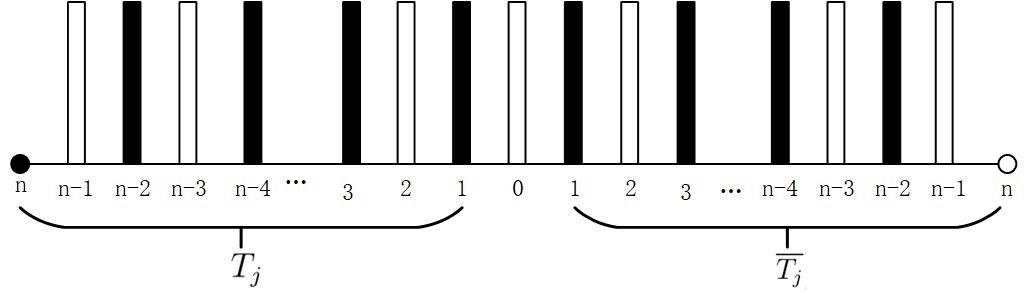}
 	\caption{the independent set $O_j$ with weight $\frac{\tau'}{t}$.}
 	\label{fig5}
 \end{figure}

 For any $I \in \mathcal{I}(G)$,
\begin{eqnarray*}
&&\sum_{k\in \{1,3, \ldots, n-2\}, j \in [s] }   \mu'(I_{k,j})  +  \sum_{k\in \{0,2, \ldots, n-1\}}\mu'(I_k)\\
&=&\left( s(\sigma'_1+\ldots + \sigma'_{n-1}) + \delta_0+\delta_2+ \ldots + \delta_{n-1} \right) \mu(I) = \mu(I),
\end{eqnarray*}
  
  and $$\sum_{j \in [s]} \mu'(O_j) = \chi_f(H) \tau'.$$
 Hence 	
 $$\sum_{J\in \mathcal{I}(\Delta_{H,n}(G))}\mu'(J)=\chi_f(G)+\chi_f(H)\tau'.$$

 It  remains to verify that $\sum_{u\in J}\mu'(J)\ge 1$ for every vertex $u$ of $\Delta_{H,n}(G)$.

An  observation we shall use frequently below is that 
\begin{equation}
\label{eqn-oj}\sum_{j \in [s]}\mu'(O_j)= \frac st \tau' = t\sigma'_0\chi_f(G).
\end{equation}

 For $u \in V(\Delta_{H,n}(G))$, let 
 \begin{center}
 	$K(u) = \{(k,j): u \in I_{k,j}\}$ and $K'(u) = \{k: u \in I_k\}$.
 \end{center}

 \bigskip
 \noindent
 
 {\bf Case 1} $u=(\star,v)$.
 
 Then $$K(u)=\{(k,j): 1\le k\le n-2, k \equiv 1\pmod{2}, j\in v\}$$ 
 and $u \in O_j$ for $j \in v$.
 By noting that $\tau' = t \sigma'_0 \frac{\chi_f(G)}{\chi_f(H)}$, we have
 \begin{equation*}
 \begin{split}
 \sum_{u\in J}\mu'(J)&=\sum_{I\in \mathcal{I}(G)}t(\sigma'_1+\sigma'_2+\ldots+\sigma'_{n-1})\mu(I) +  \tau'\\
 &=t(\sigma'_1+\sigma'_2+\ldots+\sigma'_{n-1})\chi_f(G)+t \sigma'_0 \frac{\chi_f(G)}{\chi_f(H)}\\
 &= (t\sigma'_0+s(\sigma'_1+\sigma'_2+\ldots+\sigma'_{n-1})) \frac{\chi_f(G)}{\chi_f(H)}\\
 &=1.
 \end{split}
 \end{equation*}
 

 \bigskip
 \noindent
 {\bf Case 2} $u=(x,0)$.
 
 In this case,
 $K(u)=\{ (k,j):1\le k\equiv 1\pmod{2} \le n-2, x\in I,j\in[s]\}$ and $K'(u)=\{k:0\le k\equiv 0\pmod{2}\le n-1\}$ for $x \in I$. Therefore
 \begin{equation*}
 \begin{split}
 \sum_{u\in J, J \in \mathcal{I}(\Delta_{H,n}(G))}\mu'(J)&=\sum_{x\in I, I\in \mathcal{I}(G)}\left(s(\sigma'_1+\ldots+\sigma'_{n-1})+(\delta_0+\delta_2+\ldots+\delta_{n-1}) \right)\mu(I)\\
 &= 1.
 \end{split}
 \end{equation*}

 \bigskip
 \noindent
 {\bf Case 3} $u=((x, i),v)$, $i \neq 0$.

 \bigskip
 \noindent
 {\bf Case 3(i)}   $i$ is even.
 
In this case, 
 	\begin{eqnarray*}
 		K(u) &=& \{(k,j):  i+1\le k \le n-2, k \equiv 1 \pmod{2}, x\in I, j\in [s]\} \cup \{(k,j):k=i-1, j\notin v, x\in I\}\\
 		&\cup& \{ (k,j): 1 \le k \le i-3,  k \equiv 1 \pmod{2}, j\notin v\}, \\
 		K'(u)&=&\{k:i\le k \le n-1, k \equiv 0 \pmod{2}, x\in I\}\\
 		& \cup& \{ k:  0\le k \le i-2, k \equiv 0 \pmod{2}\}.
 	\end{eqnarray*}		
 Hence   
 \begin{equation*}
 \begin{split}
 \sum_{u\in J,J \in \mathcal{I}(\Delta_{H,n}(G))}\mu'(J)&=(s-t)\sum_{I\in \mathcal{I}(G)}(\sigma'_{1} + \ldots+\sigma'_{i-2})\mu(I)+(s-t)\sum_{x \in I, I\in \mathcal{I}(G)}(\sigma'_{i-1}+\sigma'_{i})\mu(I)\\
 &+s\sum_{x \in I, I\in \mathcal{I}(G)}(\sigma'_{i+1}+\sigma'_{i+2}+\ldots+\sigma'_{n-1})\mu(I)\\
 &+\sum_{ I\in \mathcal{I}(G)}( \delta_0+\delta_2+\ldots+\delta_{i-2})\mu(I) + \sum_{x \in I, I\in \mathcal{I}(G)}(\delta_i+\delta_{i+2}+\ldots+\delta_{n-1})\mu(I)\\
 &=(s-t)(\sigma'_1+\sigma'_2+\ldots+\sigma'_{i-2})\chi_f(G)+(s-t)(\sigma'_{i-1}+\sigma'_{i})\\
 &+s(\sigma'_{i+1}+\sigma'_{i+2}+\ldots+\sigma'_{n-1})\\
 &+(\delta_0+\delta_2+\ldots+\delta_{i-2})\chi_f(G) + (\delta_i+\delta_{i+2}+\ldots+\delta_{n-1})\\
  &= (s-t)(\sigma'_1+\sigma'_2+\ldots+\sigma'_{i-1})\chi_f(G)+s(\sigma'_{i+1}+\sigma'_{i+2}+\ldots+\sigma'_{n-1}) \\
  &+(\delta_0+\delta_2+\ldots+\delta_{i-2})\chi_f(G) + (\delta_i+\delta_{i+2}+\ldots+\delta_{n-1})\\
 &=s(\sigma'_{1}+\sigma'_{2} + \ldots+\sigma'_{n-1}) + (\delta_0 + \delta_2 +\ldots + \delta_{n-1})\\
 &= 1.
 \end{split}
 \end{equation*}	
 	
 The third equality follows from the fact that $\sigma'_{i-1}+\sigma'_i = \sigma'_{i-1}\chi_f(G)$, and the fourth equality follows from (\ref{eqn-complex}) and the last equality follows from (\ref{eqnsigmadelta}).

 \bigskip
 \noindent
 {\bf Case 3(ii)}  $i>0$ is odd.
 
 Then 
 	\begin{eqnarray*}
 		K(u) &=& \{(k,j):  i\le k \le n-2, k \equiv 1 \pmod{2}, x\in I, j\in [s]\}\\
 		&\cup& \{ (k,j): 1 \le k \le i-2,  k \equiv 1 \pmod{2}, j\in v\}, \\
 		K'(u)&=&\{k:i+1\le k \le n-1, k \equiv 0 \pmod{2}, x\in I\},
 	\end{eqnarray*}	
 
 	and $u \in O_j$ for $j \in [s]$.
 Hence
 \begin{equation*}
 \begin{split}
 \sum_{u\in J,J \in \mathcal{I}(\Delta_{H,n}(G))}\mu'(J)&=t\sum_{I\in \mathcal{I}(G)}(\sigma'_{1} + \ldots+\sigma'_{i-1})\mu(I)+s\sum_{x \in I, I\in \mathcal{I}(G)}(\sigma'_{i}+\sigma'_{i+1}+\ldots+\sigma'_{n-1})\mu(I)+ \frac s t \tau'
 \\
 &+(\delta_{i+1}+\delta_{i+3}+\ldots+\delta_{n-1})\\
 &=t(\sigma'_1+\ldots+\sigma'_{i-1}) \chi_f(G) +s(\sigma'_{i}+\sigma'_{i+1}+\ldots+\sigma'_{n-1})+ t\sigma'_0\chi_f(G)\\
  &+(\delta_{i+1}+\delta_{i+3}+\ldots+\delta_{n-1})\\
 &= s(\sigma'_{1}+\sigma'_{2} + \ldots+\sigma'_{n-1}) + \delta_0+\delta_2+ \ldots + \delta_{n-1}\\
 &= 1.  
 \end{split}
 \end{equation*}
The second  equality follows from (\ref{eqn-oj}) and the third equality follows from (\ref{eqn8}) and the last equality follows from (\ref{eqnsigmadelta}).
 
 This completes the proof of   Theorem \ref{thm-main} for odd $n$.

\section{The case when $h$ is not a constant function}

We have determined the fractional chromatic number of $\Delta_{H,h}(G)$ for graphs $G,H$ with $\chi_f(H) \le \chi_f(G)$ and with $h: V(H) \to \mathbb{N}$ be a constant map, namely $h(v)=n$ for some positive integer $n$, for all $v$.

It remains as an open problem to determine the fractional chromatic number of $\Delta_{H,h}(G)$, for $h:  V(H) \to \mathbb{N}$ be an arbitrary mapping. The following lemma is easy.

\begin{lemma}
	\label{lem-hh}
	Assume $G$ and $H$ are graphs and $h,h': V(H) \to \mathbb{N}$ are two mappings with $h(v) \le h'(v)$ for all $v \in V(H)$. Then $\Delta_{H,h'}(G) \to \Delta_{H,h}(G)$, and consequently
	$\chi_f(\Delta_{H,h'}(G)) \le \chi_f(\Delta_{H,h}(G))$.
\end{lemma}
\begin{proof}
	It suffices to consider the case that $h'=h$ except that for one vertex $u$, $h'(u)=h(u)+1$. 
	Let 
	$\phi: V(\Delta_{H,h'}(G)) \to V(\Delta_{H,h}(G))$ be defined as 
	\[
	\phi(w) = \begin{cases} ((x,i-1),u), &\text{ if $w=((x,i),u)$ and $1 \le i \le h'(u)$}, \cr
	w, &\text{ otherwise}. \cr 
	\end{cases}
	\]
	It is straightforward to verify that $\phi$ is a homomorphism from $\Delta_{H,h'}(G)$ to $\Delta_{H,h}(G)$.
\end{proof}

\begin{corollary}
	Assume $G,H$ are graphs with $\chi_f(H) \le \chi_f(G)$, and $h:V(H) \to \mathbb{N}$ is a mapping. Let $n = \min \{h(v): v \in V(H)\}$ and let $X = \{v \in V(H): h(v)=n\}$. Then 
	$$\chi_f(\Delta_{H[X], n}(G)) \le \chi_f(\Delta_{H,h}(G)) \le \chi_f(\Delta_{H, n}(G)).$$
	In particular, if $n$ is even, then $\chi_f(\Delta_{H,h}(G))=\chi_f(G)+\tau(G,n)$.
\end{corollary}

A natural question is whether the upper bound or lower bound is closer to the true value. The following result shows that if $H=k_2$, then the lower bound is attained.

\begin{theorem}
	For any graph $G$ with at least one edge, if $h: V(K_2) \to \mathbb{N}$ is a mapping with $h(v_2) \ge h(v_1)=n$, then 
	$$\chi_f(\Delta_{K_2,h}(G))= \chi_f(\Delta_{n}(G)).$$
\end{theorem}
\begin{proof}
	By Theorem \ref{thm-main} and Lemma \ref{lem-hh}, it suffices to show that if $h(v_2)=h(v_1)+1=n+1$, then $\Delta_{K_2,h}(G) \to \Delta_n(G)$. 
	
	Let $\phi: V(\Delta_{K_2,h}(G)) \to V(\Delta_{n}(G))$ be defined as follows:
	\[
	\phi(w) = \begin{cases} (x,0), &\text{ if $w = (x,0)$}, \cr (x,i), &\text{ if $1 \le i \le n-1$, $w\in \{((x,i),v_1),((x,i),v_2)\}$}, \cr
	\star,   &\text{ if  $w=(\star,v_1)$  or $w=((x,n),v_2)$}, \cr 
	(x,n-1), &\text{ if $w = (\star, v_2)$, where $x$ is an arbitrary vertex of $G$}, \cr
	\end{cases}
	\]
	It is easy to check that $\phi$ is a homomorphism from $\Delta_{K_2,h}(G)$ to $\Delta_{ n}(G)$.  
\end{proof}

\subsection{The chromatic number of $\Delta_{H, h}(G)$}

It is easy to see  that   $\chi(\Delta_{H,1}(G))=\chi(G)+\chi(H)$
for any graphs $G$ and $H$. In general, we have the following proposition.

\begin{proposition}
Let $X=\{v: h(v)=1\}$. Assume $\chi(H-X) \le \chi(G)$.  Then  $$\chi(\Delta_{H,h}(G)) \le \chi(G)+ \chi(H[X]) +1.$$
\end{proposition}
\begin{proof}
	Assume $\chi(G)=k$ and $\chi(H[X])=k'$.
We colour vertices in the base of $\Delta_{H,h}(G)$ by  colours $\{1,2,\ldots, k\}$, and colour
$\{(\star, v): v \in X\}$ by colours $\{k+1,k+2, \ldots, k+k'\}$. For $v \in V(H)-X$,   we  colour vertices $((x,i), v)$ for $1 \le i\le h(v)-1$ by colour $k+k'+1$ if $i$ is odd, and by $k$ if $i$ is even. Let $\phi$ be a $k$-colouring of $H-X$ with colours $\{1,2,\ldots,k\}$, then we colour $(\star, v)$ by colour $\phi(v)$, unless $h(v)$ is odd and $\phi(v)=k$, and in which case, we colour $(\star, v)$ by colour $k+k'+1$. 
	\end{proof}

Thus if $n \ge 2$ and $\chi(\Delta_n(G)) = \chi(G)+1 > \chi(H)$, then we have  $\chi(\Delta_{H,h}(G)) = \chi(\Delta_n(G)) = \chi(G)+1$, provided that $h(v)\le n$ for some $v$, and $h(u) >1$ for all $u \in V(H)$.
 
  Assume $\chi(\Delta_n(G)) = \chi(G)$, $\chi(H) \le \chi(G)$ and $h: V(H) \to \mathbb{N}$ satisfies $h(v) \ge n$ for all $v$. A natural question is whether  equality $\chi(\Delta_{H, h}(G)) =\chi(\Delta_n(G))=\chi(G)$ always hold?

In the following, we construct a graph $G$ for which $  \chi(\Delta_n(G))=\chi(G)$, however, 
$\chi(\Delta_{K_2,n}(G)) = \chi(\Delta_n(G))+1$. 

Let $G$ be the circulant graph $G = C_7^2$ whose vertices are the integers modulo $7$, where two vertices $u, v$ are adjacent if $u-v \in \{\pm1, \pm2 \}$. It was shown in \cite{Tardif} that  $\chi(G)=\chi(\Delta_3(G))=4$. In the following, we show that   $\chi(\Delta_{K_2,3}(G))=5$. 

Assume to the contrary that $\Delta_{K_2,3}(G)$ has a proper $4$-colouring $\varphi$. 
Let 
$$C_0=[ \varphi(1,0), \varphi(2,0),\ldots,\varphi(7,0)].$$
For $i, j= 1,  2$,
let 
$$C_{i,j}=[ \varphi((1,i),v_j), \varphi((2,i), v_j),\ldots,\varphi((7,i), v_j))].$$
Up to a permutation of the colours and an automorphism of $G$, there is only one way to colour $C_0$ (which induces a copy of $G$):  $C_0=[1,2,3,1,2,3,4]$. 

Then it is easy to verify that $C_{1,j}$ has exactly four possibilities,  namely $$ C_{1,j} \in \{ [1,2,3,1,2,3,4], [1,2,3,4,2,3,4], [1,2,4,1,2,3,4], [1,2,4,4,2,3,4]\}.$$ In the first three cases, it is easy to verify that $C_{2,j}$ contains all the 4 colours $1,2,3,4$. But then there is no colour for   $(\star,v_j)$, a contradiction.

Assume $C_{1,j}=[1,2,4,4,2,3,4]$.  To save one colour for $(\star,v_j)$,  there is only one choice for $C_{2,j}$, i.e.,  $C_{2,j} = [1,3,3,1,1,3,4]$. Hence, $\varphi(\star,v_j)=2$. But then we have $\varphi(\star,v_1) =\varphi(\star,v_2)$, a contradiction.

\section*{Acknowledgement}

The first submitted version of the paper studies the so called ``double cone over a graph $G$'', which is the special case of $\Delta_{H,h}(G)$ for $H=K_2$. A referee  suggested us to study the more general graph   $\Delta_{H,h}(G)$, and commented that our result on $\Delta_{K_2,n}(G)$ for even integer $n$ remains true for $\Delta_{H,n}(G)$, provided that $\chi_f(H) \le \chi_f(G)$. We thank the referee for this nice suggestion and for many insightful comments.

\end{document}